\documentclass[11pt]{amsart}
\usepackage{amssymb,latexsym}
\usepackage{amsfonts}
\usepackage{amsmath}
\usepackage[colorlinks,linkcolor=blue,anchorcolor=blue,citecolor=blue]{hyperref}
\usepackage{algorithm}
\usepackage{algpseudocode}
\usepackage{verbatim}
\usepackage{graphicx}

\newcommand\C{{\mathbb C}}

\newcommand\Q{{\mathbb Q}}

\newcommand\Z{{\mathbb Z}}

\newcommand\al{\alpha}
\newcommand\be{\beta}

\newtheorem{theorem}{Theorem}[section]
\newtheorem{lemma}[theorem]{Lemma}

\newtheorem{corollary}[theorem]{Corollary}

\theoremstyle{definition}

\theoremstyle{remark}

\numberwithin{equation}{section}

\begin{document}

\title[Integer polynomials with several roots of 
maximal modulus]{Counting integer polynomials with several roots of maximal modulus}

\author{Art\= uras Dubickas}
\address{Institute of Mathematics, Faculty of Mathematics and Informatics, Vilnius University, Naugarduko 24,
LT-03225 Vilnius, Lithuania}
\email{arturas.dubickas@mif.vu.lt}

\author{Min Sha}
\address{School of Mathematical Sciences, South China Normal University, Guangzhou, 510631, China}
\email{min.sha@m.scnu.edu.cn}




\begin{abstract}
In this paper, for positive integers $H$ and $k \leq n$, we obtain some estimates on the cardinality of the set of monic integer polynomials of
degree $n$ and height bounded by $H$ with exactly $k$ roots of maximal modulus. These include lower and upper bounds in terms of $H$ for fixed $k$ and $n$. We also count reducible and irreducible polynomials in that set separately. Our results imply, for instance, that 
the number of monic integer irreducible polynomials of degree $n$ and height at most $H$ whose all $n$ roots have equal moduli is approximately $2H$ for odd $n$, 
while
for even $n$ there are more than $H^{n/8}$ of such polynomials.
\end{abstract}

\maketitle

\section{Introduction}

\subsection{Background and motivation} 
Counting integer polynomials with respect to some arithmetic properties and under some measures has a long history and is still very active; see, for instance, \cite{Aki, Du2014, Du2016, DS1, DS2, DS3, DS4, Waerden} and the references therein. These arithmetic properties include reducibility, degeneracy, decomposability, moduli of roots, signature of roots, multiplicative dependence of roots, etc. The relevant measures include height, Mahler measure, house, and other useful measures. These counting problems not only are of independent interest but also arise naturally from other topics. 

Recall that every linear recurrence sequence (LRS) over the rational numbers $\Q$, denoted by $\{s_m\}_{m \ge 0}$, of order $n \ge 2$ is defined by the recurrence relation
\begin{equation*}
s_{m+n}=a_{n-1}s_{m+n-1}+\cdots+ a_1s_{m+1} + a_0s_m, \quad m=0,1,2,\ldots,
\end{equation*}
where $a_0,\dots,a_{n-1}\in \Q$ and $a_0 \ne 0$.
The characteristic polynomial of this linear recurrence sequence is
$$
f(x)=x^n-a_{n-1}x^{n-1}-\cdots - a_0 \in \Q[x].
$$
The roots of $f(x)$ are called the characteristic roots of this LRS. 

The famous Skolem-Mahler-Lech Theorem states that the set of zeros of a linear recurrence
sequence, that is, the set $\{m: \, s_m=0\}$, is the union of a finite set and finitely many arithmetic progressions. However, the corresponding algorithmic 
 question, called the Skolem Problem, which asks to determine whether a given LRS has a zero term, is still widely open. 

So far, the decidability of the Skolem Problem is only known in some special cases, based on the
relative order of the absolute values of the characteristic roots. 
We say that a characteristic root is \textit{dominant} if its absolute value is maximal among all the characteristic roots. The decidability is known when there are at most $3$ (without multiplicity) dominant characteristic roots, and also for such sequences
of order at most $4$ (see \cite{MST, Ver}). However, the Skolem Problem for LRS of order at least 5 is still not decidable. 
We refer to \cite{Bilu, Lipton, Luca1, Luca2, Luca3} for some recent work
on this.

Hence, it is of interest to count integer polynomials according to their dominant roots, from which one can see how often the Skolem Problem for a random LRS is decidable. 
In \cite[Theorem 1.1]{DS1}, we showed that most of the monic integer polynomials have exactly one dominant root (counted with multiplicity).  
In \cite[Theorem 1.1]{DS3}, we also showed that most of the integer polynomials (not necessarily monic) have exactly one or two dominant roots (counted with multiplicity).
Combining this with \cite{DS2}, the second named author pointed out in \cite{Sha} essentially that a random LRS is usually covered by the known cases and so its Skolem Problem is usually decidable.

In this paper, we want to further count monic integer polynomials with several dominant roots (again counted with multiplicity).

\subsection{Main results}
Recall that for a polynomial $f$ of degree $n \ge 1$ over the complex numbers $\C$,
the largest modulus of the roots of $f$ is called the {\it house} of $f$ or the \textit{inclusion radius} of $f$, and we denote it by  
$$
r(f):= \max_{1 \leq i\leq n}|\al_i|, 
$$
where $\alpha_1, \ldots, \alpha_n$ are all the roots of $f$ (not necessarily distinct). 
Then, the roots of $f$ with modulus $r(f)$ are said to be dominant, or, alternatively, these roots are said to be of maximal modulus. 

We also recall that the \textit{height} of $f$ is defined to be the maximum of the absolute values of its coefficients. 

For any positive integers $k \leq n$ and $H$,
let $D_n(k,H)$ be the number of monic integer polynomials $f$ of degree $n$ and height at most $H$ with exactly $k$ (not necessarily distinct) roots of maximal modulus. 
In this paper, we are interested in the size of the quantity $D_n(k,H)$. 

Here and below, 
$U \ll V$ or, equivalently, $V \gg U$ for two positive quantities $U,V$ 
(depending on $k,n,H$)
means that
$U \leq c V$ for some positive constant $c$ which may depend on $k$ and $n$ but does not depend on $H$. 
Besides, $U \asymp V$ means $U \ll V \ll U$, and $U \sim V$ means $\lim_{H \to \infty} U/V=1$.

Since there are $(2H+1)^n$ monic integer polynomials of degree $n$ and with height at most $H$, 
we clearly have
\begin{equation*}\label{klop1}
(2H+1)^n=\sum_{k=1}^n D_n(k,H).
\end{equation*}

The case $k=1$ has already been investigated 
in \cite[Theorem 1.1]{DS1}, where we determined the 
asymptotics of the quantity $D_n(1,H)$. We showed there that 
\begin{equation}\label{klop111}
D_n(1,H) \sim (2H)^n
\end{equation}
as $H \to \infty$. 
This together with Lemma~\ref{lem:reducible} below implies that most of the polynomials contributing to $D_n(1,H)$ are irreducible. 

For $k=2$, in \cite[Corollary 3.3]{DS3} (the relevant notation there is $A_n(2,H)$) we showed that
\begin{equation} \label{k=2}
    D_n(2,H) \asymp H^{n-1/2}.
\end{equation}
From \eqref{k=2} and Lemma~\ref{lem:reducible} we 
see that most of the polynomials contributing to $D_n(2,H)$ are  also irreducible. 

In addition, in \cite[Theorem 3.2]{DS3} (the relevant notation there is $A_n(k,H)$) we showed that
\begin{equation} \label{k>2}
    \sum_{k=3}^{n}D_n(k,H) \ll H^{n-1}.
\end{equation}

For $k \geq 3$, our main result is the following theorem, which implies the inequality
$$
\sum_{k=3}^{n}D_n(k,H) \ll H^{n-3/2}
$$
and so gives a refinement of \eqref{k>2}. We also construct many polynomials with exactly 
$k$ dominant roots and obtain nontrivial lower bounds for $D_n(k,H)$ when $k \ge 3$. (Here and below all the constants implied in $\ll$ depend on $k$ and $n$ but not on $H$.)

  \begin{theorem}\label{thm:Dn}
 For any integers $n \ge k \ge 3$ and $H \ge 1$, we have
 \begin{equation}\label{as1}
    H^{\frac{n+1}{2}-k+\frac{5k^2-4k+7}{8n} } \ll D_n(k,H) \ll H^{n-\frac{k+1}{2}} 	\end{equation}
for $k$ odd, and 
\begin{equation}\label{as2}
    H^{\frac{n+1}{2}-k+\frac{5k^2-2k+16}{8n}} \ll D_n(k,H) \ll H^{n-\frac{k-1}{2}}
    \end{equation}
    for $k$ even. 
 \end{theorem}

In order to investigate $D_n(k,H)$ in more detail, we define $I_n(k,H)$ (resp. $R_n(k,H)$) to be the number of monic irreducible (resp. reducible) integer polynomials $f$ of degree $n$ and height at most $H$ and with exactly $k$ (not necessarily distinct) roots of maximal modulus. Clearly, 
\begin{equation}   \label{eq:Dn}
D_n(k,H) = I_n(k,H) + R_n(k,H). 
\end{equation}

Below, we will separately consider the sizes of $I_n(k,H)$ and $R_n(k,H)$. 
For $I_n(k,H)$ we have the following:

\begin{theorem}\label{thm:In}
Let $n \geq 2$ and $H \ge 1$ be two integers. Then, for any odd integer $k$ satisfying $1 \le k \le n$, we have $I_n(k,H) \sim (2H)^{n/k}$ as $H \to \infty$ if $k \mid n$ and
$I_n(k,H)=0$ if $k \nmid n$. 
Also, for any integer $k$ with $1 \le k \le n$, we have
\begin{equation*}
I_n(k,H) \ll H^{n-(k-1)/2}.
\end{equation*}
Finally, for even $n \ge 2$, we have 
\begin{equation}\label{pu1}
I_n(n,H) \gg H^{\frac{n}{8}+\frac{2}{n}+\frac{1}{4}}.
\end{equation}
\end{theorem}

Note that for odd $n$ one has $I_n(n,H) \sim 2H$ as $H \to \infty$ (compare with \eqref{pu1}).
We also remark that if $k$ is even and satisfies $k \nmid n$, it is not always true that $I_n(k,H)=0$. This is different from the case when $k$ is odd. This happens when $k=2$ and $n>2$ is odd (see Theorem~\ref{thm:InRn}), and it may also happen for $k>2$. 
For example, the  polynomial $x^6 - x^4 - 2x^3 + x^2 + x + 1$ is irreducible and has exactly $4$ roots of maximal modulus (in this example, $n=6$, $k=4$, and so $4 \nmid 6$). 

For $R_n(k,H)$, defining
\begin{equation}\label{eq:lower}
e(n,k) = \left\{\begin{array}{ll}
\frac{n+1}{2}-k+\frac{5k^2-4k+7}{8n}  & \textrm{if $k$ is odd},\\
\\
\frac{n+1}{2}-k+\frac{5k^2-2k+16}{8n}  & \textrm{if $k$ is even}
\end{array}\right.
\end{equation}
(see the exponents in the left hand sides of \eqref{as1}, \eqref{as2}), we will prove the following:

\begin{theorem}\label{thm:Rn}
For any integer $H \ge 1$, we have 
$$
R_2(1,H) \asymp H\log H, \quad R_2(2,H) \asymp H^{1/2}.
$$
For any integers $n\ge 3$ and $k$ with $1 \le k \le n$, we have 
\begin{equation*}
 R_n(k,H) \ll H^{n-(k+1)/2},
\end{equation*}
and also in all those cases, except when $k=n$ is even, 
\begin{equation}\label{pu2}
 R_n(k,H) \gg H^{e(n,k)}.  
\end{equation}
Finally, in the case when $k=n$ is even, we have 
\begin{equation}\label{pu3}
 R_n(n,H) \gg H^{\frac{n}{8}+\frac{2}{n}-\frac{1}{4}}.  
\end{equation}
\end{theorem}

For any fixed odd $k \ge 3$ (especially for $k=3$) and large enough $n$, Theorems~\ref{thm:In} and \ref{thm:Rn} show that most of the polynomials contributing to $D_n(k,H)$ are reducible, which is very much different from the cases when $k =1$ or $2$. 

To see how large the exponent $e(n,k)$ is, we can view it as a quadratic polynomial in $k$. By \eqref{eq:lower}, it can be easily verified that
\begin{equation*}
e(n,k) \ge \left\{\begin{array}{ll}
\frac{n+1}{10}  + \frac{31}{40n}  & \textrm{if $k$ is odd},\\
\\
\frac{n+3}{10} + \frac{79}{40n}  & \textrm{if $k$ is even}. 
\end{array}\right.
\end{equation*} 
This together with Theorem~\ref{thm:In} implies that for any odd $k \ge 11$ most of the polynomials contributing to $D_n(k,H)$ are reducible.

Below, in Theorem~\ref{thm:InRn} we collect the bounds for the cases when $k \le 2$ or $n \le 4$. 
It suggests that in Theorem~\ref{thm:In}, the upper bound $H^{n-(k-1)/2}$ for $I_n(k,H)$ is optimal when $k=1$ or $2$, but this is not true when $k=4$ (see the bound for $I_4(4,H)$). 
It also suggests that the upper bound $H^{n-(k+1)/2}$ for $R_n(k,H)$ in Theorem~\ref{thm:Rn} is optimal when $n \ge 3$ and $k=1$ or $2$, but this is not true when $k=3$ or $4$ (see the bounds for $R_3(3,H)$, $R_4(3,H)$ and $R_4(4,H)$). 

\begin{theorem}  \label{thm:InRn}
For any integers $n\ge 3$ and $H \ge 1$, we have 
\begin{align*}
I_n(1,H) &\sim (2H)^n,  &  R_n(1,H) &\asymp H^{n-1}, \\
I_n(2,H) &\asymp H^{n-1/2},  &  R_n(2,H) &\asymp H^{n-3/2}, \\
 I_3(3,H) &\sim 2H,  & R_3(3,H) &\asymp H^{2/3}, \\
 I_4(3,H) &=0,   & R_4(3,H) &\asymp H\log H, \\ 
 I_4(4,H) &\asymp H^{3/2},  & R_4(4,H) &\asymp H. 
\end{align*}  
\end{theorem}

We remark that the bounds for $R_2(1,H)$ and $R_2(2,H)$ have been stated in Theorem~\ref{thm:Rn}, and Theorems~\ref{thm:Dn} and \ref{thm:In} imply $I_2(1,H) \sim (2H)^2$ and $I_n(2,H) \asymp H^{3/2}$. 

Combining the last three lines of Theorem~\ref{thm:InRn} with \eqref{eq:Dn} we obtain the following corollary: 

\begin{corollary}
For any integer $H \ge 1$, we have 
\begin{equation*}
   D_3(3,H) \asymp H, \qquad D_4(3,H) \asymp H\log H, \qquad D_4(4,H) \asymp H^{3/2}. 
\end{equation*}
\end{corollary}

Finally, we point out that Theorems~\ref{thm:In} and \ref{thm:Rn} imply Theorem~\ref{thm:Dn}. Indeed, for $k$ even the upper bound on $D_n(k,H)$ stated in Theorem~\ref{thm:Dn} (see \eqref{as2}) follows directly from \eqref{eq:Dn} and Theorems~\ref{thm:In} and \ref{thm:Rn} (see also Lemma~\ref{uU}). 
For $k$ odd the upper bound on $D_n(k,H)$ (see \eqref{as1}) follows from \eqref{eq:Dn} and Theorems~\ref{thm:In} and \ref{thm:Rn}, because $n/k \leq n-(k+1)/2$. 
The lower bounds on $D_n(k,H)$ as in \eqref{as1} and \eqref{as2} follow from 
\eqref{eq:Dn}, \eqref{pu1} and \eqref{pu2}. (Note that \eqref{pu1} gives a better exponent than that in \eqref{pu3}.) 
This completes the proof of Theorem~\ref{thm:Dn}.

In all that follows we first gather some auxiliary results in Section~\ref{sec:pre} and then
prove Theorems~\ref{thm:In}, \ref{thm:Rn} and \ref{thm:InRn}
in Sections~\ref{sec:proof1}, \ref{sec:proof2} and \ref{sec:proof3} respectively.

\section{Preliminaries}\label{sec:pre}

Recall that the {\it height} of a degree $n \geq 1$ polynomial 
\[f(x)=a_n x^n + a_{n-1} x^{n-1} +\dots + a_1 x +a_0=a_n \prod_{i=1}^n (x-\al_i) \in \C[x], a_n \ne 0,\]
is defined by \[H(f):=\max_{0 \leq i \leq n} |a_i|,\] and its {\it Mahler measure}
by 
\[M(f):=|a_n| \prod_{i=1}^n \max(1,|\al_i|).\]
The following inequalities between these two quantities are well-known:
\begin{equation}\label{ine5}
	2^{-n} H(f) \leq M(f) \leq \sqrt{n+1} \> H(f), 
	\end{equation}
see, for instance, \cite[(3.12)]{Wald}. 
So, for a fixed $n$, one has
\begin{equation}\label{eq:Mahler}
H(f) \ll M(f) \ll H(f).
\end{equation}
If $f$ can be factored as the product of two non-constant polynomials $g,h\in \C[X]$ (that is, $f=gh$), 
then, by the definition of the Mahler measure, we have 
$$
M(f)=M(g)M(h).
$$ 
So, combined  with \eqref{eq:Mahler} this yields
\begin{equation}   \label{eq:height}
H(g)H(h) \ll H(f) \ll H(g)H(h). 
\end{equation}

In \cite{Waerden}, van der Waerden proved the following result about counting monic reducible polynomials (see also \cite[Lemma 6]{Du2016} or \cite[Lemma 2.2]{DS4}). 
It implies that the number of monic irreducible integer polynomials of degree $n \ge 2$ and of height at most $H \ge 1$ is asymptotic to $(2H)^n$ as $H \to \infty$. 

\begin{lemma}\label{lem:reducible}
For integers $n\ge 2$ and $m \ge 1$ with $1 \le m \le n/2$, let $F_{n}(m,H)$ be the number of monic integer polynomials of degree $n$ and height at most $H$ which are reducible over the rational integers $\Z$ with a factor (not necessarily irreducible) of degree $m$. 
Then, if $1 \le m < n/2$, we have
$$
 F_{n}(m,H) \asymp H^{n-m},
$$
while if $m=n/2$, we have 
$$
F_{n}(n/2,H) \asymp H^{n/2}\log H.
$$
\end{lemma} 

The following lemma is \cite[Theorem 3.2]{Aki}.

\begin{lemma}\label{akis}
	Let $s \geq 0$ and $n \geq 1$ be two integers 
	satisfying $s \leq n/2$. 
     Then, there are constants $v_1(s,n)>0$ and $v_2(s,n)>0$ such that 
for each $B > 0$ the number $J_n(s,B)$ of monic irreducible integer polynomials $f$
     of degree $n$ with $r(f) \leq B$ and with exactly $2s$ non-real roots 
     satisfies
	\[
		\left| J_n(s,B)-v_1(s,n) B^{n(n+1)/2} \right| \le v_2(s,n) B^{n(n+1)/2-1}.
	\]
\end{lemma}

The next lemma proved by Ferguson \cite{Ferguson} is about irreducible polynomials with many roots of equal modulus. 
It is a generalization of Boyd's result \cite{Boyd} about irreducible polynomials with many roots of maximal modulus.

\begin{lemma}  \label{lem:Ferguson}
Suppose that an irreducible integer polynomial $f(x) \in \Z[x]$ has $k$ roots, at least one real, on the circle $|z| = c>0$. Then, $f(x) = g(x^k)$, where $g(x)$ has no more than one real root on any circle in $\C$ with center at the origin.
\end{lemma}

The following result is a corollary of Capelli's lemma about the irreducibility of compositions of polynomials. 

\begin{lemma} \label{lem:Capelli}
 Let $g(x)$ be a monic irreducible integer polynomial of positive degree such that $|g(0)|^{1/m} \not\in \Q$ for any integer $m > 1$. Then, $g(x^k)$ is also irreducible in $\Z[x]$ for any integer $k \ge 1$. 
\end{lemma}

\begin{proof} 
First, we note that for any monic integer polynomial, its irreducibility in the ring $\Z[x]$ is equivalent to that over the rational numbers $\Q$. 
Since $g(x)$ is irreducible, by Capelli's lemma, we know that $g(x^k)$ is irreducible over $\Q$ if and only if $x^k - \alpha$ is irreducible over $\Q(\alpha)$ for any root $\alpha$ of $g(x)$ (note that $\Q(\alpha)$ is a subfield of the complex numbers $\C$).  

Now, by contradiction, we suppose that $g(x^k)$ is reducible in $\Z[x]$ for some integer $k \ge 1$. 
Then, $x^k - \alpha$ is reducible over $\Q(\alpha)$ for some root $\alpha$ of $g(x)$. 
Using a classical result also due to Capelli (see \cite[Theorem 19]{Sch}), we deduce that either $\alpha=\beta^p$ for some $\beta \in \Q(\alpha)$ and some prime divisor $p$ of $k$, 
 or $4 \mid k$ and $\alpha=-4\beta^4$ for some $\beta \in \Q(\alpha)$. 
So, there exists a polynomial $h(x) \in \Q[x]$ such that either $\alpha = h(\alpha)^p$ or $\alpha = -4h(\alpha)^4$. 
Note that $g(x)$ is the minimal polynomial of $\alpha$ over $\Q$, and so either $g(x) \mid h(x)^p-x$ or $g(x) \mid 4h(x)^4 +x$. Let $\alpha_1, \ldots, \alpha_d$ be all the roots of $g(x)$ (not necessarily distinct). 
Then, either 
$$
\alpha_1 \cdots \alpha_d = h(\alpha_1)^p \cdots h(\alpha_d)^p = \left(h(\alpha_1) \cdots h(\alpha_d) \right)^p, 
$$
or 
$$
\alpha_1 \cdots \alpha_d = (-4h(\alpha_1)^4) \cdots (-4h(\alpha_d)^4) = (-1)^d 2^{2d}\left(h(\alpha_1) \cdots h(\alpha_d) \right)^4. 
$$
So, noticing $g(0)=(-1)^d \alpha_1 \cdots \alpha_d$ and $h(\alpha_1) \cdots h(\alpha_d) \in \Q$, we have either $|g(0)|^{1/p} \in \Q$ or $|g(0)|^{1/2} \in \Q$. This contradicts with the assumption that $|g(0)|^{1/m} \not\in \Q$ for any integer $m > 1$. 

Hence, we may conclude that $g(x^k)$ is irreducible in $\Z[x]$ for any integer $k \ge 1$. 
\end{proof}

We conclude this section with the following upper bound:

\begin{lemma} \label{uU}
For any integers $n \geq k \geq 1$ and $H \geq 1$ we have
\begin{equation}\label{uUU} 
D_n(k, H) \ll H^{n-(k-1)/2}. 
\end{equation}
\end{lemma}

\begin{proof} 

Consider a monic integer polynomial 
\[
f(x)=x^n+a_{n-1}x^{n-1}+\dots+a_1x+a_0,
\]
with $a_{n-1}, \dots, a_1, a_0 \in \Z \cap [-H,H]$.  Assume that $f$ has exactly $k$ roots of maximal modulus.
 Then, by \eqref{ine5}, we obtain 
\[r(f)^k \leq M(f) \leq \sqrt{n+1} \> H(f) \leq \sqrt{n+1} H,\]
which gives 
\begin{equation} \label{eq:r(f)}
r(f) \ll H^{1/k}.  
\end{equation}
From  $|a_{n-i}| \leq \binom{n}{i} r(f)^i$ for any integer $i$ with $1 \le i \le n$, we see that 
the vector 
\[(a_{n-1},a_{n-2},\dots,a_{n-k})\]
can take at most
$\ll r(f)^{1+2+\dots+k}\ll r(f)^{k(k+1)/2} \ll H^{(k+1)/2}$ values.
Since the vector
\[(a_{n-k-1},a_{n-k-2},\dots,a_{0})\]
takes at most $(2H+1)^{n-k} \ll H^{n-k}$
values, there are at most
\[\ll H^{(k+1)/2+n-k} \ll H^{n-(k-1)/2}\]
suitable monic polynomials $f$. This proves
\eqref{uUU}.
\end{proof}

\section{Proof of Theorem~\ref{thm:In}} \label{sec:proof1}

By \eqref{eq:Dn} and Lemma~\ref{uU}, we deduce
\begin{equation}  \label{eq:In}
I_n(k, H) \le D_n(k, H) \ll H^{n-(k-1)/2},
\end{equation} 
 as claimed in the theorem. 

Now, we assume that $k$ is odd. 
For any polynomial $f(x)$ contributing to $I_n(k,H)$, 
by Lemma~\ref{lem:Ferguson}, we know that 
$f(x) = g(x^k)$, where $g(x)$ is a monic irreducible integer polynomial of degree $n/k$ and height at most $H$ and with exactly one root of maximal modulus (that is, $g(x)$ contributes to $I_{n/k}(1,H)$). So, we must have $k \mid n$, which implies $I_n(k, H) = 0$ if $k \nmid n$.  
Conversely, if $g(x)$ is a polynomial as the above, then $g(x^k)$ is a monic integer polynomial of degree $n$ and height at most $H$ and with exactly $k$ roots of maximal modulus. It remains to count the number of such polynomials $g(x)$ for which $g(x^k)$ is also irreducible. By Lemma~\ref{lem:Capelli}, when $|g(0)| \ne b^m$ for any positive integers $b,m \ge 2$, $g(x^k)$ is irreducible. It is easy to see that the number of such polynomials $g(x)$ for which  $|g(0)| = b^m$ for some positive integers $b,m \ge 2$ is at most 
$$
(2H+1)^{n/k - 1} \cdot 2\sum_{b=2}^{\sqrt{H}} \log_{b} H \ll H^{n/k-1/2}\log H. 
$$
Hence, the number of such polynomials $g(x)$ for which $g(x^k)$ is also irreducible is asymptotic to $(2H)^{n/k}$ as $H \to \infty$. Therefore, for $k \mid n$, we deduce
$$
I_n(k,H) \sim (2H)^{n/k}
$$
as $H \to \infty$.

In the case when $n \ge 2$ is even, the desired lower bound for $I_n(n,H)$ will be proved below in \eqref{eq:Inn}.

\section{Proof of Theorem~\ref{thm:Rn}} \label{sec:proof2}
\subsection{Proof of the upper bound}

We will show that \[R_n(k,H)\ll H^{n -(k+1)/2}.\] 

First, we consider the case when $n=2$. 
Note that a polynomial $f(x)$ contributing to $R_2(2,H)$ must be of the form $f(x)=(x-a)^2$ or $f(x)=x^2-a^2$. 
So, we have $R_2(2,H) \asymp H^{1/2}$. Then, the bound $R_2(1,H) \asymp H\log H$ follows from Lemma~\ref{lem:reducible}.

Assume that $n \geq 3$.
The case when $k=1$ is trivial by Lemma~\ref{lem:reducible}, so assume that $k \ge 2$.
By induction on $n$ or $k$, it suffices to consider those polynomials $f(x)$ contributing to $R_n(k,H)$ of the form 
\begin{equation}  \label{eq:f1f2}
f(x)=f_1(x)f_2(x),
\end{equation}
 where both $f_1$ and $f_2$ are monic irreducible integer polynomials. We further assume that $\deg f_i = n_i$,  and $f_i$ has exactly $k_i$ roots of modulus $r(f)$, $i=1,2$. So, 
$$
n_1 \ge 1, \quad  n_2 \ge 1, \quad n = n_1 + n_2,  \quad k = k_1 + k_2. 
$$
We denote by $C_n(n_1, k_1, H)$ the number of such polynomials $f(x)$  of the form \eqref{eq:f1f2} contributing to $R_n(k,H)$. 
Our aim is to show the inequality
\begin{equation}\label{as4}
C_n(n_1,k_1,H) \ll H^{n-(k+1)/2}. 
\end{equation}

We first consider the case when $k_1k_2=0$. 
Without loss of generality, we assume $k_1=0$  (so, $k_2=k$). 
By \eqref{eq:height} and \eqref{eq:In} and considering possible choices of $f_1,f_2$ (with $H(f_1) = b$ and $H(f_2)\ll H/b$ for $1 \le b \le H$), we obtain 
\begin{equation*}
\ll \sum_{b=1}^{H} b^{n_1} \Big(\frac{H}{b}\Big)^{n_2 - (k_2-1)/2} = H^{n_2 - (k-1)/2} \sum_{b=1}^{H}\frac{1}{b^{n_2 - n_1 - (k-1)/2}}
\end{equation*}
polynomials of the form \eqref{eq:f1f2}.
Consequently,
\begin{equation}  \label{eq:n10}
C_n(n_1, 0, H) \ll  H^{n_2 - (k-1)/2} \sum_{b=1}^{H}\frac{1}{b^{n_2 - n_1 - (k-1)/2}}.
\end{equation}
If $n_2 - n_1 - (k-1)/2 > 1$, then 
\begin{equation*}
C_n(n_1, 0, H) \ll H^{n_2 - (k-1)/2} \le H^{n - (k+1)/2}
\end{equation*}
which is \eqref{as4}.
If $n_2 - n_1 - (k-1)/2 = 1$, then $k$ must be odd (so $k \ge 3$). Now, if $n_1 \ge 2$, then from \eqref{eq:n10} we have 
\begin{equation*}
C_n(n_1, 0, H) \ll H^{n_2 - (k-1)/2}\sum_{b=1}^{H} \frac{1}{b} \ll   H^{n_2 - (k-1)/2} \log H \ll  H^{n - (k+1)/2},
\end{equation*}
which is \eqref{as4}.
If $n_1=1$, then, observing that $n_2 = 2+ (k-1)/2$ and $n_2 \ge k \ge 3$, we deduce $k=3$ and hence  $n_2 = 3$. Thus, $f_2(x)=x^3 +a$ for some integer $a$ by Lemma~\ref{lem:Ferguson}. This implies 
\begin{equation*}
C_n(n_1, 0, H) \ll H^2 = H^{n - (k+1)/2},
\end{equation*}
as required. 

It remains to consider the case $n_2 - n_1 - (k-1)/2 < 1$. If $n_2 \ge 3$, then from \eqref{eq:n10} and noticing that $n_2 \ge k$ we deduce
\begin{equation*}
C_n(n_1, 0, H) \ll H^{n_2 - (k-1)/2} \cdot H^{1-(n_2 - n_1 - (k-1)/2)} = H^{n_1 +1} \le H^{n - (k+1)/2},
\end{equation*}
which is \eqref{as4}.
If $n_2=2$, then we must have $k=2$ (note that we have assumed $k \ge 2$), and then if moreover $n_1 \ge 2$, 
we have (considering $H(f_2)= b$ and $H(f_1)\ll H/b$ for $1 \le b \le H$)
\begin{equation*}
C_n(n_1, 0, H) \ll \sum_{b=1}^{H} (H/b)^{n_1} b^{n_2-(k_2-1)/2} = H^{n_1} \sum_{b=1}^{H} \frac{1}{b^{n_1-3/2}} \ll H^{n_1+1/2}, 
\end{equation*}
which yields \eqref{as4} due to $n_1+1/2=n-(k+1)/2$.
It remains to consider the subcase when $n_1=1, n_2=2$ and $k_2=k=2$, for which by \eqref{eq:In} we deduce (as $n=n_1+n_2=3$)
\begin{equation*}
C_n(n_1, 0, H) \ll \sum_{b=1}^{H} (H/b)^{n_2 - (k_2-1)/2} = H^{3/2} \sum_{b=1}^{H} \frac{1}{b^{3/2}} \ll H^{3/2} = H^{n-(k+1)/2}. 
\end{equation*}
Therefore, when $k_1=0$, we always have \eqref{as4}. 

We now consider the case when $k_1k_2 > 0$. Assume first that either $n_1=1$ or $n_2=1$. 
Without loss of generality, we assume that $n_1=1$. 
 Then, $k_1=n_1=1$, $n_2 = n-1$, $k_2=k-1$. 
Noticing $r(f) \ll H^{1/k}$ by \eqref{eq:r(f)}, we know that $f_1(x)$ is of the form $x-b$ with $|b|\ll H^{1/k}$, 
and when $b$ is fixed, we have $r(f_2)=r(f)=b$ and $H(f_2) \ll H/b$. So, considering possible choices of $f_2$, we obtain 
\begin{equation}\label{eq:n11}
C_n(1, 1, H) \ll \sum_{b=1}^{\lfloor H^{1/k} \rfloor} b^{1+2+\cdots + k_2} \Big(\frac{H}{b}\Big)^{n_2 - k_2} = H^{n - k} \sum_{b=1}^{\lfloor H^{1/k} \rfloor}\frac{1}{b^{n-\frac{k(k+1)}{2}}}.
\end{equation}
Noticing $ n - k < n- (k+1)/2$ (since $k \ge 2$), we find that that in the case when $n - k(k+1)/2 \ge 1$, 
from \eqref{eq:n11} the inequality  
\begin{equation*}
C_n(1, 1, H) \ll H^{n-k}\log H \ll  H^{n - (k+1)/2} 
\end{equation*}
is true.
Otherwise, we have
\begin{equation*}
C_n(1, 1, H) \ll H^{n - k+\frac{1}{k}\big(1-n + k+ \frac{k(k -1)}{2}\big)} = H^{n-\frac{k+1}{2} + \frac{k+1-n}{k}}, 
\end{equation*}
which, together with another assumption $k \le n-1$, gives 
\eqref{as4}
again. 
So, it remains to consider the subcase when $n_1=k_1 = 1$ and $k=n$, so that $n_2=k_2 = n-1$. In this subcase, such polynomials $f(x)$ are of the form $f(x)=(x-b)f_2(x)$ with $b \in \Z$ and $r(f_2)=r(f)=|b| \ll H^{1/n}$ by \eqref{eq:r(f)}, and then noticing $k=n$ we know that the constant term of $f_2$ is fixed up to a sign if $b$ is fixed,  and so we have 
\begin{equation*}
C_n(1, 1, H) \ll \sum_{b=1}^{H^{1/n}} b^{1+2+\cdots + (n-2)} \ll H^{\frac{1}{n}\big(1+\frac{(n-1)(n-2)}{2}\big)}= H^{\frac{n-1}{2}+\frac{2}{n} - 1}, 
\end{equation*}
which yields \eqref{as4}.

Now, we consider the final case when $k_1k_2 > 0$ and $n_1, n_2 \ge 2$. 
Since 
$$H(f_1)H(f_2) \ll H(f) \le H,$$
 we have $\min\{H(f_1), H(f_2)\} \ll H^{1/2}$. 
Without loss of generality, we assume $H(f_1) \ll H^{1/2}$. 
So, $r(f)=r(f_1) \ll H^{1/(2k_1)}$ by applying \eqref{eq:r(f)} to $f_1$. 
Then, by \eqref{eq:In} and considering possible choices of $f_1, f_2$, we obtain 
\begin{equation*}
\begin{split}
C_n(n_1, k_1, H) & \ll H^{\frac{1}{2}(n_1 - \frac{k_1 -1}{2})} \cdot H^{n_2 - \frac{k_2-1}{2}-1} \cdot H^{\frac{1}{2k_1}} \\
& =  H^{\frac{1}{2}n_1 + n_2 - \frac{1}{4}k_1- \frac{1}{2}k_2 - \frac{1}{4} + \frac{1}{2k_1}}. 
\end{split}
\end{equation*}
Note that since $n_1 \ge 2$ and $n_1 \ge k_1 \ge 1$, we have 
\begin{equation*}
\begin{split}
& n - \frac{k+1}{2} - \left(\frac{1}{2}n_1 + n_2 - \frac{1}{4}k_1- \frac{1}{2}k_2 - \frac{1}{4} + \frac{1}{2k_1}\right) \\ 
& \quad  = \frac{1}{4}\left(2n_1 - k_1 - 1 - \frac{2}{k_1}\right) \ge 0.
\end{split}
\end{equation*}
Hence, we have \eqref{as4} again.

Note that the number of choices of $(n_1,k_1)$ depends only on $n$ and $k$. 
So \eqref{as4} yields the desired upper bound
$R_n(k,H)  \ll H^{n -(k+1)/2}$.

\subsection{Proof of the lower bound}
Now, we want to prove the claimed lower bound 
\[R_n(k,H) \gg H^{e(n,k)}.\] 

Assume first that the integer $k$ satisfying $2 \leq k \leq n$ is even. Set $\ell=k/2$.  Fix an integer $m$ 
	in the range 
	\begin{equation}\label{mm1}
		 \frac{H^{2/n}}{5} \leq m \leq \frac{H^{2/n}}{4}.
		\end{equation}
	Consider
	a monic integer irreducible polynomial $(x-\beta_1) \dots (x-\beta_{\ell})$ of degree $\ell$ with all $\ell$ real roots 
	in the interval $[-2\sqrt{m},2\sqrt{m}]$. 
	By Lemma~\ref{akis}, there are 
	$\gg m^{\ell(\ell+1)/4}$ of such polynomials. 
	Note that we can consider only polynomials with all roots in the open interval $(-2\sqrt{m},2\sqrt{m})$, since the root $\be_i$ may coincide with the endpoints 
	of the interval $\pm 2\sqrt{m}$ only if $\ell=2$, but this removes at most one polynomial from those $\gg m^{\ell(\ell+1)/4}$ above. 
	
	Now, for each
	of the above polynomials with all roots in the open interval $(-2\sqrt{m},2\sqrt{m})$, we set
	\begin{equation}\label{gg1}
g(x)=(x^2-\be_1 x+m) \dots (x^2-\be_{\ell}x+m) \in \Z[x].
\end{equation}
	From $\be_i^2-4m<0$, we see that each such polynomial $g$ has all its $2\ell=k$ roots on the circle 
	$|z|=\sqrt{m}$. Hence, its Mahler measure $M(g) = \sqrt{m}^{2\ell}=m^{\ell}$. 
	
	Consider monic polynomials
	$f$ of the form
	\begin{equation}\label{uu0}
f(x)=g(x) h(x),
\end{equation} 
	where  $g$ is as in \eqref{gg1} and $h \in \Z[x]$ is a monic polynomial of degree $n-k$ that
	has its all $n-k$ roots in the disc $|z| \leq \sqrt{m}/2$.
	Then, $M(h) \leq (\sqrt{m}/2)^{n-k}$.   Consequently, by \eqref{ine5}, \eqref{uu0}, the multiplicativity of the Mahler measure and the upper bound on $m$ in \eqref{mm1}, we deduce
	\begin{align*}
H(f) &=H(gh) \leq 2^n M(gh)=2^n M(g)M(h) \leq 2^n m^{\ell} (\sqrt{m}/2)^{n-k} \\& =2^k m^{n/2} \leq 2^n m^{n/2}=(4m)^{n/2} \leq H.
\end{align*}
So, if $k < n$, such polynomials $f$ constructed in \eqref{uu0} indeed contribute to $R_n(k,H)$. 
	By Lemma~\ref{akis}, the number of such monic irreducible polynomials $h$ of degree $n-k$ is $\gg m^{(n-k)(n-k+1)/4}$. So, by the lower bound on $m$ in
	\eqref{mm1}, 
	 the number of pairs $(g,h)$ for each $m$ satisfying \eqref{mm1}
	is \[\gg m^{\ell(\ell+1)/4+(n-k)(n-k+1)/4} \gg H^{\ell(\ell+1)/(2n)+(n-k)(n-k+1)/(2n)}.\] 
	By \eqref{mm1}, the number of $m$'s is greater than $0.04 H^{2/n}$ for $H$ large enough. This gives the lower bound 
	\begin{equation}\label{uu2}\gg H^{2/n+\ell(\ell+1)/(2n)+(n-k)(n-k+1)/(2n)}\end{equation}
	for the number of distinct pairs $(g,h)$. 
	
	Now, we will prove that two distinct pairs $(g_1,h_1)$ and $(g_2,h_2)$ give distinct products
	$g_1h_1$ and $g_2 h_2$. We first show that
	each polynomial defined in \eqref{gg1} is irreducible. 
	Indeed, each factor $x^2-\beta_i x+m$ has two complex conjugate roots, so if $g$ were reducible, we would have $\prod_{i \in S} (x^2-\beta_ix+m) \in \Z[x]$ for some nonempty proper subset $S$ of $\{1,2,\dots,\ell\}$. This is only possible if the elementary symmetric polynomials in $\beta_i,$ $i \in S$, are all rational. But then $\prod_{i \in  S}(x-\beta_i) \in \Z[x]$, contrary to the irreducibility of the polynomial $\prod_{i=1}^{\ell}(x-\beta_i)$.  
	Assume that 
	\begin{equation}\label{koko}g_1h_1=g_2h_2\end{equation}
	for some $(g_1,h_1) \ne (g_2,h_2)$. 
	Then, $g_1 \ne g_2$.
	By \eqref{gg1}, the constant coefficients of $g_1$ and $g_2$ are $m_1^{\ell}$ and $m_2^{\ell}$ respectively for some integers $m_1, m_2$.
	Without restriction of generality, we may assume that $m_1 \leq m_2$. 
	Since $g_1 \ne g_2$ and $g_1,h_1,g_2,h_2$ are all irreducible, we must have
	$g_1=h_2$. (Of course, by the degree consideration, this can only happen if $2\ell=k=n-k$). Because of \eqref{koko}, this yields $g_2=h_1$. However, the roots of $g_2$ are all of moduli $\sqrt{m_2}$, while the roots of the polynomial $h_1$, by its construction, are all in the disc $|z| \leq \sqrt{m_1}/2<\sqrt{m_2}$, which is a contradiction. 
	Hence, \eqref{koko} is impossible for $(g_1,h_1) \ne (g_2,h_2)$. 
	
	It follows that the number of monic integer polynomials $f$ as in \eqref{uu0} with exactly $k=2\ell$ roots in the circle $|z|=r(f)$ satisfying $H(f) \leq H$ is bounded from below as in \eqref{uu2}. The exponent $e(n,k)$ for $H$
	there equals 
	\begin{align*}
e(n,k) &= 
\frac{2}{n}+\frac{\ell(\ell+1)}{2n}+ \frac{(n-k)(n-k+1)}{2n}	\\&=\frac{2}{n}
+\frac{k(k+2)}{8n}+\frac{n^2+k^2-2nk+n-k}{2n} \\&=\frac{k^2+2k+16}{8n}+\frac{n+1}{2}-k+\frac{4k^2-4k}{8n} \\&=
\frac{n+1}{2}-k+\frac{5k^2-2k+16}{8n},
\end{align*}
which is exactly as in \eqref{eq:lower} when $k$ is even and $k < n$. 

When $k$ is even and $k=n$, in the construction \eqref{uu0} $h(x)$ is of degree $n-k=0$, and so $h(x)=1$, 
and then $f(x)=g(x)$ is irreducible and contributes to $I_n(n,H)$. Hence, $I_n(n,H) \gg H^{e(n,n)}$, that is, 
\begin{equation}  \label{eq:Inn}
I_n(n,H) \gg H^{\frac{n}{8}+\frac{2}{n}+\frac{1}{4}}, 
\end{equation}
which gives the lower bound for $I_n(n,H)$ in Theorem \ref{thm:In} when $n$ is even. 

Hence, if $k$ is even, it remains to bound $R_n(n,H)$ where $k=n$. Since $n \ge 3$ and $n$ is even in this case, we have $k=n \ge 4$. As in the construction \eqref{uu0}, we first construct such a polynomial $g(x)$ as in \eqref{gg1} with $\ell=(n-2)/2$ and then let $h(x)=x^2 - m$, and so $f(x)=g(x)h(x)$ indeed contributes to $R_n(n,H)$, where one can verify $H(f) \le H$ as before. 
Hence, for each $m$ in the range \eqref{mm1}, 
	 the number of distinct pairs $(g,h)$ (that is, the number of such polynomials $f$) is 
$$
\gg m^{\ell(\ell +1)/4} \gg H^{\frac{2}{n} \cdot \frac{n(n-2)}{16}} = H^{\frac{n}{8} - \frac{1}{4}}, 
$$
and thus we have  
$$
R_n(n,H) \gg H^{\frac{2}{n}} \cdot H^{\frac{n}{8} - \frac{1}{4}} = 
H^{\frac{n}{8} + \frac{2}{n} - \frac{1}{4}}, 
$$
which is exactly as in \eqref{pu3} when $k$ is even and $k=n$.

Now, we turn to the case of odd $k$ satisfying $1 \leq k \leq n$. Set $\ell=(k-1)/2$. This time 
we argue with integer $m$ in the range
	\begin{equation}\label{mm11}
	\frac{H^{1/n}}{3} \leq m \leq \frac{H^{1/n}}{2}.
\end{equation}
As above we now consider
a degree $\ell$ monic integer irreducible polynomial $(x-\beta_1) \dots (x-\beta_{\ell})$, but with all $\ell$ real roots 
in the open interval $(-2m,2m)$. By Lemma~\ref{akis} (and the explanation in the case when $k$ is even), there are 
$\gg m^{\ell(\ell+1)/2}$ of such polynomials. For each
of them, we set
\begin{equation}\label{gg11}
g(x)=(x-m)(x^2-\be_1 x+m^2) \dots (x^2-\be_{\ell}x+m^2) \in \Z[x].
\end{equation}
By $\be_i^2-4m^2<0$, $i=1, \ldots, \ell$, we see that the polynomial $g$ has all its $k=2\ell+1$ roots on the circle 
$|z|=m$, with $m$ being the only real root among them. Note that if $k=1$, we have $g(x)=x-m$.

Consider monic polynomials $f$ as in \eqref{uu0} with $g$ as in \eqref{gg11} and with monic $h \in \Z[x]$
of degree $n-k$ with all $n-k$ roots in the disc $|z| \leq m/2$. Then,
$M(h) \leq (m/2)^{n-k}$ and similarly to the above, 
by \eqref{ine5}, \eqref{mm11} and $M(g)=m^k$,
we find that
	\begin{align*}
H(f) &=H(gh) \leq 2^n M(gh)=2^n M(g)M(h) \leq 2^n m^{k} (m/2)^{n-k} \\& =2^k m^{n} \leq (2m)^{n} \leq H.
\end{align*}
So, such polynomials $f$ indeed contribute to $R_n(k,H)$ (note that $g(x)$ is reducible in $\Z[x]$ if $k > 1$). 
This time, by Lemma~\ref{akis}, the number of such monic irreducible polynomials $h$ of degree $n-k$ is $\gg m^{(n-k)(n-k+1)/2}$. Hence, by the lower bound on $m$ in
\eqref{mm11}, 
the number of pairs $(g,h)$ for each $m$ satisfying \eqref{mm11}
is \[\gg m^{\ell(\ell+1)/2+(n-k)(n-k+1)/2} \gg H^{\ell(\ell+1)/(2n)+(n-k)(n-k+1)/(2n)}.\] 
Next, by \eqref{mm11}, the number of $m$'s is greater than $0.15 H^{1/n}$ for $H$ large enough. This gives the lower bound 
\begin{equation}\label{uu22}
\gg H^{1/n+\ell(\ell+1)/(2n)+(n-k)(n-k+1)/(2n)}
\end{equation}
for the number of distinct pairs $(g,h)$. 

Now, exactly as in the case with $k$ even, we can see that the polynomial
$g(x)/(x-m) \in \Z[x]$ is irreducible, so for $(g_1,h_1) \ne (g_2,h_2)$ 
and $g_1h_1=g_2h_2$ we must have $g_1(x)/(x-m_1)=h_2(x)$ and similarly
$g_2(x)/(x-m_2)=h_1(x)$. This yields $m_1=m_2$ and the same contradiction as the above for $k$ even. Therefore, the number of suitable monic polynomials $f$ with exactly
$k$ roots on $|z|=r(f)$ and $H(f) \leq H$ is bounded below as in \eqref{uu22}. In view of $\ell=(k-1)/2$ the exponent there is 
	\begin{align*}e(n,k) &= 
	\frac{1}{n}+\frac{\ell(\ell+1)}{2n}+ \frac{(n-k)(n-k+1)}{2n}	\\&=\frac{1}{n}
	+\frac{k^2-1}{8n}+\frac{n^2+k^2-2nk+n-k}{2n} \\&=\frac{k^2+7}{8n}+\frac{n+1}{2}-k+\frac{4k^2-4k}{8n} \\&=
	\frac{n+1}{2}-k+\frac{5k^2-4k+7}{8n},
\end{align*}
which is exactly as claimed in \eqref{eq:lower} when $k$ is odd. 
This completes the proof for odd $k \geq 1$ and the proof of the theorem. 

\section{Proof of Theorem~\ref{thm:InRn}} \label{sec:proof3}

\subsection{Proof of the bounds for $k \le 3$}
For $I_n(1,H)$, from Theorem~\ref{thm:In} we directly obtain 
$$
I_n(1,H) \sim (2H)^n.
$$

For $R_n(1,H)$, since $n \ge 3$, by Lemma~\ref{lem:reducible} we first have 
$$
R_n(1,H) \ll H^{n-1},
$$
and then considering polynomials of the form $(x+1)g(x)$ with $|g(0)| > 1$ and $g(x)$ contributed to $I_{n-1}(1,H/2)$, we know that all these polynomials contribute to $R_n(1,H)$, and so we get 
$$
R_n(1,H) \gg H^{n-1},
$$
and thus we obtain 
$$
R_n(1,H) \asymp H^{n-1}.
$$

For $I_n(2,H)$, combining Theorem~\ref{thm:Dn} with Lemma~\ref{lem:reducible}, we directly obtain 
$$
I_n(2,H) \asymp H^{n-1/2}. 
$$

For $R_n(2,H)$, since $n \ge 3$, by Theorem~\ref{thm:Rn} we have 
$$
R_n(2,H) \ll H^{n-3/2}, 
$$
and then considering polynomials of the form $(x+1)g(x)$ with $|g(0)| > 1$ and $g(x)$ contributed to $I_{n-1}(2,H/2)$, we know that all these polynomials contribute to $R_n(2,H)$, and so we get 
$$
R_n(2,H) \gg H^{n-1-1/2} = H^{n-3/2},
$$
and thus we obtain 
$$
R_n(2,H) \asymp H^{n-3/2}.
$$

The desired results $I_3(3,H) \sim 2H$ and $I_4(3,H)=0$ follow directly from Theorem~\ref{thm:In}. 

For $R_3(3,H)$, let $f(x)$ be a polynomial contributing to $R_3(3,H)$. If $f$ can be factored as the product of three polynomials of degree one in $\Z[x]$, then $f$ must be of the form $(x\pm a)(x \pm a)(x \pm a), a \in \Z$, and so, noticing $|a^3| \le H$ we have that the number of such polynomials $f$ in this case is $\ll H^{1/3}$. If $f$ is the product of a polynomial of degree one and an irreducible polynomial of degree two in $\Z[x]$, then $f$ must be of the form 
$$
f(x) = (x+a)(x^2+bx+a^2) = x^3+(a+b)x^2+(a^2+ab)x +a^3,
$$
with $a,b\in \Z$ such that $b^2-4a^2 <0$, and so we have $|a| \le H^{1/3}$ and $|b| < 2|a|$, and then considering the choices of $(a,b)$ we deduce that the number of such polynomials $f$ in this case is $\asymp H^{2/3}$. Hence, we obtain 
$$
R_3(3,H) \asymp H^{2/3}, 
$$
as claimed. 

Finally, we want to prove $R_4(3,H) \asymp H\log H$.
In the following, let $f(x)$ be a polynomial contributing to $R_4(3,H)$. 

If $f(x)$ is of the form 
$$
f(x)=(x+a)g(x), \quad a \in \Z, 
$$
with irreducible polynomial $g(x) \in \Z[x]$ of degree 3 such that $r(f)=r(g)=|a|$. Then, $g(x)$ has two conjugate non-real roots of modulus $|a|$, and so by looking at the constant term of $f$ we know that $g$ has a root in $\Q$, which contradicts with the irreducibility of $g$. Hence, this case does not happen. 

If $f(x)$ is of the form 
$$
f(x)=(x+a)g(x), \quad a \in \Z, 
$$
with irreducible polynomial $g(x) \in \Z[x]$  of degree 3 such that $g(x)$ has exactly three roots of modulus $r(f) > |a|$. 
Then, by Lemma~\ref{lem:Ferguson} we know that $g(x)$ is of the form $x^3+b$, and moreover $|a|^3 < |b|$. 
Noticing in this case $f = x^4+ax^3+bx+ab$, we have $|a|^4 < |ab| \le H$. So, the number of such polynomials $f$ is 
\begin{equation} \label{eq:R43-1}
    \asymp \sum_{a=1}^{H^{1/4}} H/a \asymp H\log H.
\end{equation}

If $f(x)$ is of the form 
$$
f(x)=g(x)h(x), 
$$
with irreducible polynomials $g(x), h(x) \in \Z[x]$ of degree 2 such that $g$ has two roots of modulus $r(f)$ and $h$ has exactly one root of modulus $r(f)$. Then, $g$ only has non-real roots and $h$ only has real roots. When $g$ is fixed, then the real root of $h$ having modulus $r(f)$ is fixed up to a sign, and so $h$ has at most two choices (noticing that $h$ is irreducible). In addition, we have $r(g) = r(f) \ll H^{1/3}$ by \eqref{eq:r(f)}. Hence, by considering the choices of $g$, we deduce that  the number of such polynomials $f$ is 
\begin{equation} \label{eq:R43-2}
    \ll H^{1/3} \cdot H^{2/3} = H.
\end{equation}

If $f(x)$ is of the form 
$$
f(x)=g(x)h(x), 
$$
with irreducible polynomial $g(x)$ and reducible polynomial $h(x) \in \Z[x]$ of degree 2 such that $g$ has two roots of modulus $r(f)$ and $h$ has exactly one root of modulus $r(f)$. 
Then, $f(x)$ is in fact of the form 
$$
f(x)=(x+a)(x+b)(x^2+cx+b^2), \quad a,b,c \in \Z, 
$$
with $|a| < |b|$. 
By \eqref{eq:r(f)}, we have $r(f) = |b| \ll H^{1/3}$. So, $|a| \ll H^{1/3}$ and $|c| \le 2|b| \ll H^{1/3}$. 
Hence, by considering the choices of $a,b,c$, we deduce that  the number of such polynomials $f$ is 
\begin{equation} \label{eq:R43-3}
    \ll H^{1/3} \cdot H^{1/3} \cdot H^{1/3} = H.
\end{equation}

If $f(x)$ is of the form 
$$
f(x)=g(x)h(x), 
$$
with irreducible polynomial $g(x)$ and reducible polynomial $h(x) \in \Z[x]$ of degree 2 such that $g$ has exactly one root of modulus $r(f)$ and $h$ has two roots of modulus $r(f)$. 
Then, we know that the irreducible polynomial $g(x)$ has only real roots, and $h(x)$ is of the form $x^2 -a^2$ or $(x-a)^2$, where $a \in \Z$ and $|a|=r(f)$. 
However, a real number of modulus $|a|$ must be $a$ or $-a$, 
and so either $g(a)=0$ or $g(-a)=0$. But $g$ is irreducible in $\Z[x]$. Hence, this case can not happen. 

Finally, if $f(x)$ can be factored as the product of four polynomials in $\Z[x]$ of degree one, then noticing that $r(f) \ll H^{1/3}$ and $f$ has exactly three roots of modulus $r(f)$, we get that  the number of such polynomials $f$ is 
\begin{equation} \label{eq:R43-4}
    \ll H^{1/3} \cdot H^{1/3} = H^{2/3}.
\end{equation}

Therefore, combining \eqref{eq:R43-1}, \eqref{eq:R43-2}, \eqref{eq:R43-3} with \eqref{eq:R43-4}, we obtain 
$$
R_4(3,H) \asymp H\log H,
$$
as claimed.

\subsection{Proof of the bound for $R_4(4,H)$}

Let $f(x)$ be a polynomial contributing to $R_4(4,H)$. 

If $f(x)$ is of the form 
$$
f(x)=(x+a)g(x), \quad a \in \Z, 
$$
with irreducible polynomial $g(x) \in \Z[x]$ of degree 3 such that $r(f)=r(g)=|a|$. Then, $g(x)$ must have a real root of modulus $|a|$. However, a real number of modulus $|a|$ must be $a$ or $-a$, 
and so either $g(a)=0$ or $g(-a)=0$. But $g$ is irreducible in $\Z[x]$. Hence, this case can not happen. 

If $f(x)$ is of the form 
$$
f(x)=g(x)h(x), 
$$
with polynomials $g(x), h(x) \in \Z[x]$ of degree 2. 
Then, both $g$ and $h$ have two roots of modulus $r(f)$. 
Then, $g(x)$ and $h(x)$ must be of the forms 
$$
g(x) = (x^2+ax \pm b), \quad h(x)=(x^2+cx \pm b), \quad a,b,c\in \Z.
$$
Since $r(g) = r(h)= r(f) \ll H^{1/4}$ by \eqref{eq:r(f)}, 
we deduce that $|a| \ll H^{1/4}$, $|b| \ll H^{1/2}$ and $|c| \ll H^{1/4}$.
Hence, by considering the choices of $(a,b,c)$, we deduce that  the number of such polynomials $f$ is 
\begin{equation} \label{eq:R44-1}
    \ll H^{1/4} \cdot H^{1/2} \cdot H^{1/4} = H.
\end{equation}
In addition, all the polynomials of the form $(x^2+ax + b)(x^2+cx + b)$ with $a,b,c\in \Z$, $0 <b < H^{1/2}$, $a^2-4b<0$ and $c^2 - 4b <0$ contribute to $R_4(4,H)$ for any large enough $H$, and thus we obtain that 
the number of such polynomials $f$ is 
\begin{equation} \label{eq:R44-2}
    \gg \sum_{b=1}^{H^{1/2}} \sqrt{b} \cdot \sqrt{b} = \sum_{b=1}^{H^{1/2}} b \gg  H.
\end{equation}

If $f(x)$ can be factored as the product of four polynomials in $\Z[x]$ of degree one, then noticing that $r(f) \ll H^{1/4}$ and all the four roots of $f$ have modulus $r(f)$, we get that  the number of such polynomials $f$ is 
\begin{equation} \label{eq:R44-3}
    \ll H^{1/4}.
\end{equation}

Therefore, combining \eqref{eq:R44-1}, \eqref{eq:R44-2} with \eqref{eq:R44-3}, we obtain 
$$
R_4(4,H) \asymp H,
$$
as claimed.

\subsection{Proof of the bound for $I_4(4,H)$}

First, since $I_2(2,H) \asymp H^{3/2}$, by considering polynomials of the form $x^4+ax^2+b$ such that $x^2+ax+b$ contributes to $I_2(2,H)$ and using Lemma~\ref{lem:Capelli} as before, we obtain 
\begin{equation}  \label{eq:I44-1}
I_4(4,H) \gg H^{3/2}.
\end{equation}

Now, we want to prove that $I_4(4,H) \ll H^{3/2}$. 
Let $f(x)$ be a polynomial contributing to $I_4(4,H)$. 
So, $f(x)$ is irreducible and all of its roots have equal modulus. 

If $f(x)$ is of the form $x^4+ax^2+b$, then since $x^2+ax+b$ contributes to $D_2(2,H)$, by Theorem~\ref{thm:Dn} we get that the number of such polynomials $f$ is 
\begin{equation}  \label{eq:I44-2}
\ll H^{3/2}.
\end{equation}

In the following, we consider polynomials $f(x)$ contributing to $I_4(4,H)$ and not of the form $x^4+ax^2+b$. 
We denote by $I_4^*(4,H)$ the number of such polynomials $f$. 
So, in view of \eqref{eq:I44-2} it remains to prove $I_4^*(4,H) \ll H^{3/2}$. In fact, we can do even better. 

Note that $f$ from $I_4^*(4,H)$
cannot have a real root. Indeed, if $\alpha$ is a real root of $f$, then so must be $-\alpha$, so $f(x)=f(-x)$, which is impossible, because  
$f$ is not of the form $x^4+ax^2+b$.  
Consequently, each polynomial $f(x)$ contributing to $I_4^*(4,H)$ can be factored as 
\begin{equation*}
f(x)=(x-\alpha)(x-\bar{\alpha})(x-\beta)(x-\bar{\beta})=(x^2-ux+r)(x^2-vx+r),
\end{equation*}
where $\bar{\alpha}$ and $\bar{\beta}$ are the complex conjugates of the non-real roots $\alpha$ and $\beta$ respectively, 
$u=\alpha+\bar{\alpha}$, $v=\beta+\bar{\beta}$ and 
$|\alpha|^2=|\beta|^2=r>0$.  
The coefficients of $f$ for $x^3,x^2,x,1$ are, respectively, 
\[-(u+v), \>\> uv+2r, \>\> 
-(u+v)r, \>\> r^2 \in \Z.\] 
Note that $u+v \ne 0$, since $f \in I_4^*(4,H)$. From $u+v \in \Z \setminus \{0\}$
and $(u+v)r \in \Z$ it follows that $r \in \Q$. Hence, as $r^2 \in \Z$, we must
have $r \in \Z$. From $r>0$ we conclude that $r$ is a positive integer. Furthermore, $r \leq \sqrt{H}$, since the height of $f$ does not exceed $H$.

Now, for each such positive integer $r$, it remains to count the number of pairs of integers $(u+v,uv)$ in terms of $r$ and then sum  them up
for $r$ from $1$ to $\lfloor \sqrt{H} \rfloor$. 
Since the roots of $f$ are non-real, we clearly have
$u^2-4r<0$ and $v^2-4r<0$. Hence, $|u+v|<4\sqrt{r}$ and $|uv|<4r$. 
Thus, the number of integer pairs $(u+v,uv)$ does not exceed $(8\sqrt{r}+1)(8r+1)$. Consequently, 
\[I_4^*(4,H) \leq \sum_{r=1}^{\sqrt{H}} (8\sqrt{r}+1)(8r+1) \ll \sum_{r=1}^{\sqrt{H}} r^{3/2}\ll H^{5/4}.\]

Therefore, combining this estimate with \eqref{eq:I44-1}
and \eqref{eq:I44-2}, we obtain 
$$
I_4(4,H) \asymp H^{3/2},
$$
as claimed.

\section*{Acknowledgement}
The authors would like to thank Igor Shparlinski for bringing their attention back to the topic of this paper and also for helpful discussions. 
For the research, Min Sha was supported by the Guangdong Basic and Applied Basic Research Foundation (No. 2022A1515012032).

\end{document}